\documentclass{article}

\usepackage{amssymb}
\usepackage{amsmath}
\usepackage{amsthm}
\usepackage{graphicx}
\usepackage{a4wide}
\usepackage{tikz}

\renewcommand{\(}{\left(}
\renewcommand{\)}{\right)}
\newtheorem{prop}{Proposition}
\newtheorem{coroll}{Corollary}
\newtheorem{theorem}{Theorem}
\newtheorem{lemma}{Lemma}

\theoremstyle{definition}
\newtheorem{definition}{Definition}

\theoremstyle{remark}
\newtheorem{remark}{Remark}

\newtheorem{fact}{Fact}
\newtheorem{ack}{Acknowledgement\!\!}

\newcommand{\dd} {{\cal D}}
\newcommand{\bth}{\begin{theorem}}
\renewcommand{\eth}{\end{theorem}}
\newcommand{\bl}{\begin{lemma}}
\newcommand{\el}{\end{lemma}}
\newcommand{\bp}{\begin{prop}}
\newcommand{\ep}{\end{prop}}
\newcommand{\bcor}{\begin{coroll}}
\newcommand{\ecor}{\end{coroll}}
\newcommand{\bdf}{\begin{definition}}
\newcommand{\edf}{\end{definition}}
\newcommand{\bpf}{\begin{proof}}
\newcommand{\epf}{\end{proof}}
\newcommand{\brem}{\begin{remark}}
\newcommand{\erem}{\end{remark}}
 
\newcommand{\eps}{\varepsilon}

\title{Enumeration of generalized $BCI$ lambda-terms}
\author{Olivier Bodini\thanks{Institut Galil\'ee, Univ. Paris 13, Villetaneuse (France). 
Supported by ANR Magnum project (France)}, 
Dani\`ele Gardy\thanks{PRiSM, Universit\'e de Versailles Saint-Quentin, 78035 Versailles, France.
This author's work was partially carried out during her sabbatical leave at the Institute for
Discrete Mathematics and Geometry, Technische Universit\"at Wien, Austria.
Supported by ANR Boole project (France)}, 
Bernhard Gittenberger\thanks{Institute for Discrete
Mathematics and Geometry, Technische Universit\"at Wien, Wiedner Hauptstrasse 8-10/104,
A-1040 Wien, Austria. 
Supported by FWF grant SFB F50-03 and \"OAD, grant F04/2012}, 
Alice Jacquot$^\ast$}
\date{\today}

\begin{document}
\maketitle
\begin{abstract}
We investigate the asymptotic number of elements of size $n$ in a particular class of closed
lambda-terms (so-called $BCI(p)$-terms) which are related to axiom systems of combinatory logic.
By deriving a differential equation for the generating function of the counting sequence we obtain
a recurrence relation which can be solved asymptotically. We derive differential equations for the
generating functions of the counting sequences of other more general classes of terms as well: the
class of $BCK(p)$-terms and that of closed lambda-terms. Using elementary arguments we obtain
upper and lowerestimates for the number of closed lambda-terms of size $n$. Moreover, a recurrence
relation is derived which allows an efficient computation of the counting sequence. $BCK(p)$-terms
are discussed briefly. 
\end{abstract}


\section{Introduction}

Lambda-terms play a prominent role in the theory of computer programming. In order to investigate
properties of randomly generated lambda-terms we have to know how many terms of a given size there
are. This paper is devoted to the asymptotic counting of particular classes of lambda-terms. 

Lambda-terms were invented by Church and Kleene in the 30ies (see \cite{Ch36, Kl35a, Kl35b}) 
together
with a set of rules for manipulating them, the so-called lambda-calculus. This is a very powerful
formal language which can be used to describe computer programs, analyze programming
languages or investigate decision problems. Moreover, it is the basis of the programming language
LISP.

A lambda-term is a formal expression built of variables and a quantifyer 
$\lambda$ which in general occurs more than once and acts on one of the free variables.
It can be described by the context-free grammar $T::= a\;\mid\;(T*T)\;\mid\;\lambda a.T$ where a
is a variable. The concatenation of terms is called \emph{application} and adding the prefix
$\lambda a$ to a term is called \emph{abstraction}. Each abstraction \emph{binds} a
variable in the whole term following it and each variable can only be bound by at most one
abstraction. A term where all the variables are bound is called a \emph{closed} lambda-term. For
example, $(\lambda x.(x*x)*\lambda y.y)$ is a closed lambda-term whereas $(\lambda x.(x*z)*\lambda 
y.y)$ is an open one.

Our aim is to study the asymptotic number of closed lambda-terms of given size when the size is
tending to infinity. We define the size of a lambda-term recursively by
$$
|x|=1, \quad |\lambda x.T|=1+|T|, \quad |(S*T)|=1+|S|+|T|.
$$

Moreover, note that we will count lambda-terms up to isomorphism: Only the structure of the
bindings is important; variable names are unimportant. The terms $\lambda y.(\lambda x.x* \lambda
z.y)$, $\lambda y.(\lambda x.x* \lambda x.y)$, $\lambda x.(\lambda y.y* \lambda z.x)$ are
considered to be identical. Observe that the second term is obtained from the first one by replacing
$z$ by $x$, which is ``by coincidence'' the same variable as that in the subterm $\lambda x.x$
just left to it. But as stated above the important issue is that the last quantifier does not bind
the variable following it; therefore the name must only be different from $y$.

Since the determination of the asymptotic number of lambda-terms seems to be a hard problem (\emph{cf.}
the discussion of this issue in \cite{BGG11} and for a similar problem in \cite[end of
Sec.~3]{GL12}) we confine ourselves with the asymptotic analysis of
a simpler subclass of lambda-terms and give an outlook to the analysis of a larger and more
complicated subclass. The classes considered are $BCI(p)$- and $BCK(p)$-terms. The names stem from
the correspondence of $BCI(1)$- and $BCK(1)$-terms to the logical systems $BCI$ and $BCK$,
respectively, which are studied in combinatory logic (see \cite{II1,II2,KT78}). 

The plan of the paper is as follows: In the next section we state our notations, definitions and
some immediate observations. In Section~\ref{gfsec} we derive the functional equations for
the generating functions corresponding to $BCI(p)$-terms, $BCK(p)$-terms as well as general closed
lambda-terms. Then we will derive the asymptotic order of the number of $BCI(p)$-terms
(Section~\ref{bcisec}). 
Section~\ref{clolamter} is devoted to an upper and a lower estimate for the number $\lambda_n$ of
closed lambda-terms of size $n$. This is done using rather elementary arguments, but it is still
sufficient to obtain the asymptotic main term of $\log\lambda_n$. Moreover, we will derive a
recurrence relation which allows an efficient computation of the numbers $\lambda_n$. In the final
section, we briefly discuss $BCK(p)$-terms. 

The enumeration of $BCI(1)$-terms was carried out
by Bodini et al. \cite{BGJ10} by constructing a nice bijection to certain diagrams. It was shown
that the number of $BCI(1)$-terms of size~$n$ is asymptotically 
\begin{equation}
\label{bcione}
\frac{C}{n^{1/6}} \left(  \frac{2n}{e} \right)^{n/3} 
\end{equation}
if $n \equiv 2$ mod 3 and zero otherwise. They obtained also the asymptotic number of
$BCK(1)$-terms which differs from \eqref{bcione} by a factor $e^{\frac{1}{2}
(2n)^{2/3} - \frac{1}{6} (2n)^{1/3}}$.

Models with a different notion of size (leaves do not count, i.e. they have weight zero) were
studied in \cite{DGKRTZ10, GL12}. In \cite{DGKRTZ10} upper and lower bounds for the counting
sequence were derived and questions like typability were discussed. The paper \cite{GL12}
appoaches the counting problem by representations of terms using de Bruijn indices. They derive a
recurrence relations for the number of terms with or without constraints on the number of free
variables and discuss the issue of random generation of terms as well. This allows an efficient
computation and experimental analysis of term properties like typability of some shape
characteristics.

\section{Notation and basic facts}

A lambda-term can be regarded as a so-called \emph{enriched tree} which is a particular directed
acyclic graph. In fact, consider a Motzkin tree (i.e., a rooted unary-binary tree) and add
directed edges connecting a unary node and a leaf such that each leaf is ``bound'' by a directed
edge from exactly one of the unary nodes that are its ancestors in the tree. The correspondence is
obvious (see Figure~\ref{treefigure}): 
\begin{figure}[ht]
        \centering
\unitlength5mm
\begin{picture}(10,6)
        \put(0,0){\includegraphics[width=50mm,height=25mm]{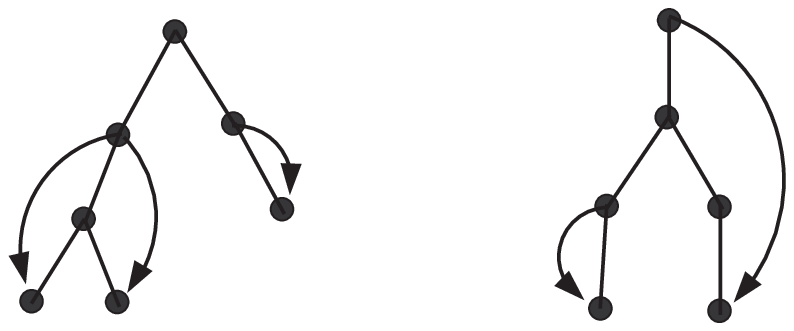}}
        \put(-1.5,0){\normalsize $(\lambda x.(x*x)*\lambda y.y)$}
        \put(6.5,0){\normalsize $\lambda y.(\lambda x.x* \lambda z.y)$}
        \put(0.5,1.2){\small $x$}
        \put(1.8,1.2){\small $x$}
        \put(3.7,2.1){\small $y$}
        \put(7.2,1.1){\small $x$}
        \put(8.6,1.1){\small $y$}
        \end{picture}
\caption{\small Two enriched trees and the closed lambda-terms corresponding to them. Note that the node
labels can be omitted, since $(\lambda x.(x*x)*\lambda y.y)$ and $(\lambda a.(a*a)*\lambda b.b)$
are the same term.}
        \label{treefigure}
\end{figure}
leaves correspond to variables, unary nodes to abstractions, binary nodes to applications
and the additional directed edges to the binding relations between abstractions and variables.
Clearly, since all leaves are bound, the lambda-term is closed. Of course, open lambda-terms can
be represented in an analogous manner by a directed acyclic graph where some leaves have in-degree
zero (that means that they have no ingoing \emph{directed} edge).

We will not distinguish between a lambda-term and its enriched tree representation. In addition, 
when speaking of lambda-terms, we will utilize the following abuse of the wording: A \emph{unary
node} of a lambda-term is a unary node (i.e. node of out-degree one) of the underlying Motzkin
tree (i.e. a node becoming unary if all directed edges are removed). These are precisely the nodes
corresponding to
abstractions. Analogously, we call the nodes corresponding to applications \emph{binary nodes} 
and nodes corresponding to variables \emph{leaves} of the lambda-term. In a strict sense, leaves
have always degree one and in-degree one as well (i.e. each leaf $x$ is incident with exactly one
\emph{undirected} and exactly one \emph{directed} edge pointing towards $x$). 

Moreover, we 
distinguish between \emph{edges}, i.e. edges of the underlying Motzkin tree, and
\emph{pointers}, i.e. directed edges from a unary node to a leaf. 

\begin{definition}

\begin{itemize}
\item
$BCI(p)$ is the set of (non-empty) closed lambda-terms where each unary node has {\em exactly} $p$
pointers, i.e. binds exactly $p$ occurrences of its variable. 
\item
$BCK(p)$ is the set of closed lambda-terms where each unary node binds {\em at most} $p$ leaves.
\end{itemize}
\end{definition}

\begin{figure}
\begin{center}
\begin{tikzpicture}[baseline=(1),inner sep=0pt, circle, minimum size=4pt, scale =0.6]

\tikzstyle{every node}+=[fill]

\node (0) at (0,0) {};
\node (1) at (0,-1) {};
\node (2) at (-1,-2) {};
\node (3) at (-1,-3)  {};
\node (4) at (-1,-4) {};
\node (5) at (1,-2) {};
\node (6) at (1,-3) {};
\node (7) at (2,-4)  {};
\node (8) at (0,-4) {};
\node (9) at (-1,-5) {};
\node (10) at (1,-5) {};
\node (11) at (1,-6)  {};
\node (12) at (0,-7) {};
\node (13) at (-1,-9) {};
\node (14) at (1,-9) {};
\node (15) at (2,-7)  {};
\node (16) at (0,-8)  {};

\draw 
(0) to (1)
(1) to (2)
(5) to (1)
(1) to (2)
(2) to (3)
(3) to (4)
(5) to (6)
(6) to (7)
(6) to (8)
(9) to (8)
(8) to (10)
(10) to (11)
(11) to (12)
(12) to (16)
(13) to (16)
(14) to (16)
(11) to (15)
 ;
\draw (4) to[dashed, bend right,<-] (2);
\draw (7) to[dashed, bend right,<-] (5);
\draw (15) to[dashed, bend right,<-] (5);
\draw (9) to[dashed,bend left=15,<-] (5);
\draw (13) to[dashed, bend left,<-] (12);
\draw (14) to[dashed, bend right,<-] (12);
\end{tikzpicture}
\hfil
\begin{tikzpicture}[baseline=(1),inner sep=0pt, circle, minimum size=4pt, scale =0.6]

\tikzstyle{every node}+=[fill]

\node (0) at (0,0) {};
\node (1) at (0,-1) {};
\node (2) at (0,-2) {};
\node (3) at (-2,-3)  {};
\node (4) at (-3,-4) {};
\node (5) at (-4,-5) {};
\node (6) at (-2,-5) {};
\node (7) at (-1,-4)  {};
\node (8) at (3,-3) {};
\node (9) at (4,-4) {};
\node (10) at (2,-4) {};
\node (11) at (1,-6)  {};
\node (12) at (3,-6) {};
\node (13) at (2,-5) {};

\draw 
(0) to (1)
(1) to (2)
(8) to (2)
(2) to (3)
(3) to (4)
(5) to (4)
(6) to (4)
(7) to (3)
(9) to (8)
(8) to (10)
(10) to (13)
(11) to (13)
(12) to (13)
 ;
\draw (5) to[dashed, bend left,<-] (0);
\draw (7) to[dashed, bend left,<-] (0);
\draw (6) to[dashed, bend left=15,<-] (1);
\draw (9) to[dashed,bend right,<-] (1);
\draw (11) to[dashed, bend left,<-] (10);
\draw (12) to[dashed, bend right,<-] (10);
\end{tikzpicture}
\hfil
\begin{tikzpicture}[baseline=(1),inner sep=0pt, circle, minimum size=4pt, scale =0.6]

\tikzstyle{every node}+=[fill]

\node (0) at (0,0) {};
\node (1) at (-1,-1) {};
\node (2) at (-1,-2) {};
\node (3) at (-2,-3)  {};
\node (4) at (-2,-4) {};
\node (5) at (-2,-5) {};
\node (6) at (-2,-6) {};
\node (7) at (-3,-7)  {};
\node (8) at (-1,-7) {};
\node (10) at (0,-3) {};
\node (11) at (0,-4)  {};
\node (12) at (1,-1) {};
\node (13) at (1,-2) {};
\node (14) at (1,-3) {};
\node (15) at (1,-4) {};
\draw 
(0) to (1)
(1) to (2)
(10) to (2)
(2) to (3)
(3) to (4)
(5) to (4)
(6) to (5)
(7) to (6)
(8) to (6)
(11) to (10)
(0) to (12)
(12) to (13)
(14) to (13)
(14) to (15)
 ;
\draw (7) to[dashed, bend left,<-] (3);
\draw (8) to[dashed, bend right,<-] (5);
\draw (11) to[dashed, bend right,<-] (10);
\draw (15) to[dashed,bend right,<-] (13);
\end{tikzpicture}

\end{center}
\caption{\small Left: a closed $\lambda$-term of size $17$. Center: a term in $BCI(2)$ of size $14$. Right: a term in $BCK(1)$ of size $15$.}
\label{definitionexamples}
\end{figure}
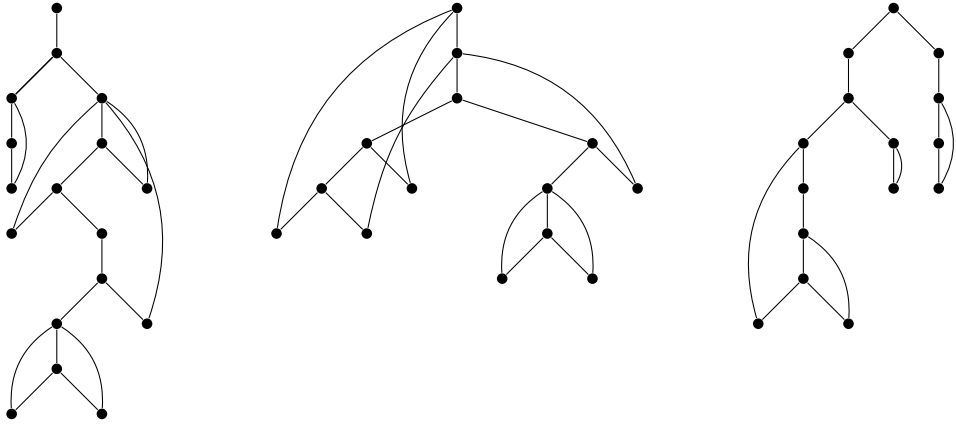

A lambda-term from $BCI(p)$ has three types of nodes: unary nodes (which are actually of arity
$p+1$, as there are $p$ pointers going from this node to leaves), binary nodes, and leaves.
The size of such a lambda-term is the total number of its nodes. We start with some obvious
observations: 

\begin{fact}
The smallest terms of $BCI(p)$ have one unary node at the root and $p$ leaves. There are $p$
pointers from the root to all the leaves. Obviously, if we remove the root and all its pointers,
we are left with a binary tree. Cleary, their size is $2p$.

The number of such terms is therefore equal to the number of binary trees with $p-1$ binary nodes
and $p$ leaves. This is precisely the Catalan number $C_{p-1}=\binom{2p-2}{p-1}/p$.
\end{fact}

\begin{fact}
A term of $BCI(p)$ with $j$ unary nodes has $pj$ leaves and $pj-1$ binary nodes; 
its size is therefore equal to $(2p+1) j -1$.
\end{fact}

\section{The generating functions for various classes of closed \\ lambda-terms} \label{gfsec}

We will enumerate lambda-terms by means of generating functions. Let $g_n=g_n^{(p)}$ be the number
of $BCI(p)$-terms of size $n$ and $G_p(z)$ be the generating function of this sequence. By Fact 2
we have actually 
$$
G_p(z) = \sum_{j\geq 1} g_{j(2p+1) -1} z^{j(2p+1) -1}.
$$
Analogously, define $F_p(z)=\sum_{n\geq 1} f_nz^n$ and $\Lambda(z)=\sum_{n\geq 1} \lambda_nz^n$ 
where
$f_n=f_n^{(p)}$ is the number of $BCK(p)$-terms of size $n$ and $\lambda_n$ the number of closed
lambda-terms of size $n$. 

The next step is the setting up of functional equations for the generating functions. This will be
done by giving a formal specification of the combinatorial objects and then using the symbolic
method (see \cite{FlSe09}). 
From \cite{BGJ10} we already know that $G_1(z)$ satisfies the equation
\[
G_1(z) = z^2 + z G_1(z)^2 + \Delta_1 G_1(z),
\]
where the differential operator $\Delta_1$ is $2 z^4 D$ and $D$ denotes the ordinary 
differential operator. 

\bp\label{BCI}
The generating function of $BCI(p)$-terms satisfies the differential equation 
\begin{equation} \label{bci}
G_p (z) = C_{p-1} z^{2p} + z G_p(z)^2 + \Delta_p G_p(z)
\end{equation} 
where 
\begin{equation}\label{defdelta}
\Delta_p=\sum_{l=1}^p \frac{\alpha_{l,p}}{l!} z^{l+2p+1} D^l
\end{equation}
with constants $\alpha_{l,p}$ defined by 
\begin{equation} \label{alpha}
\alpha_{l,p}
 = \sum_{\sum_i s_i = l;\; \sum_i i s_i = p} \binom{l}{s_1, \dots, s_p} \prod_{m=1}^p
\binom{2m}m^{s_m}.
\end{equation}
\ep

\bpf
A $BCI(p)$-term can be specified by the formal equation  
\begin{equation} \label{speci}
\cal T=\cal S\cup (\{\circ\}\times \cal T\times \cal T) \cup (\{\circ\}\times \tilde{\cal T}), 
\end{equation} 
where the set $\cal S$ is the set of all smallest $BCI(p)$-terms (\emph{cf.} Fact 1) and
$\tilde{\cal T}$ a certain set of open $BCI(p)$-terms.\footnote{Note the slight abuse of notation
in \eqref{speci}: The last cartesian product on the right-hand side is not a cartesian product in
a strict sense, but only on the level of the underlying Motzkin trees, since we will add pointers
going from the new root to some leaves of $\tilde{\cal T}$.} This can be explained as follows: A
$BCI(p)$-term falls into exactly one of three categories: It is either 
\begin{itemize}
\item a smallest term,
\item or its root is a binary node and the two subterms attached to the root are themselves $BCI(p)$-terms, 
\item or its root is a unary node and the subterm attached to the root is a $BCI(p)$-term with
exactly $p$ free leaves. 
\end{itemize}
In order to specify all $BCI(p)$-terms and avoid ambiguities, we have to take some care in the
choice of $\tilde{\cal T}$. Indeed, each $BCI(p)$-term will be generated exactly once by the
specification \eqref{speci} if we generate $\tilde{\cal T}$ by starting with a $BCI(p)$-term 
and then generating $p$ leaves and connecting them to the unary root node by a pointer in the following way. 
To construct a term $\tilde t\in \{\circ\}\times \tilde{\cal T}$, 
choose a $BCI$-term $t$ and $p$ nodes of $t$, where multiple
choices of a node are allowed. Each node $v$ corresponds to an edge, namely the edge leading to $v$
if $v$ is not the root and the edge connecting $v$ with the new root (of the term $\tilde t\in
\{\circ\}\times \tilde{\cal T}$) otherwise. Thus the choice of the $p$ nodes ``hits'' edges of the
term $\tilde t$. Assume that $l$ edges are hit and $s_i$ of them exactly $i$ times.  

If an edge is hit $i$ times, then replace it by a path where at each node of the path a binary
tree is attached, either to the left or to the right of the path, and the number of leaves of all
these binary trees altogether is equal to $i$ (see Figure~\ref{figbci} for an illustration of this
process). Thus the replacement creates $i$ new leaves and $i$
new internal nodes. The whole replacement process creates exactly $\sum_{i=1}^p is_i=p$ new leaves
and $p$ new internal nodes in $t$. Therefore $\tilde t$ has exactly $2p+1$ more nodes than $t$ and
obviously $\tilde t$ is a $BCI(p)$-term. Conversely, if we have a $BCI(p)$-term with a unary root
node, then removing it together with its pointers yields a term with $p$ free leaves. These leaves
must be children of a binary node since otherwise the parent node must have pointers to $p$
descendants which is impossible. Thus the free leaves induce a set of subtrees of $t$ which are
binary trees with free leaves only. 

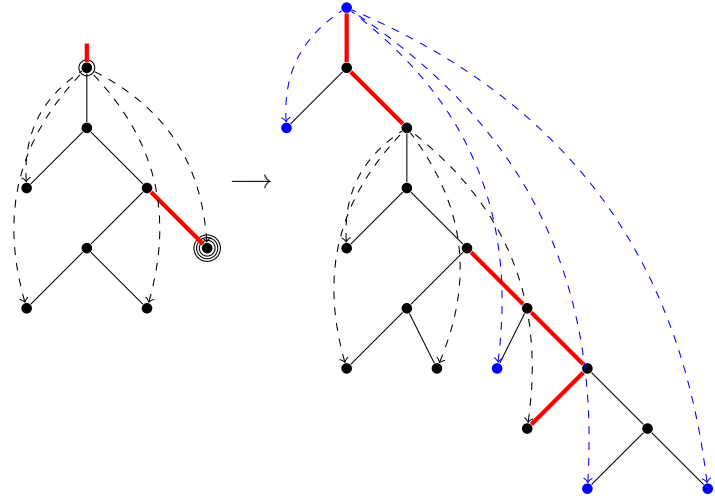
\begin{figure}[ht]
\begin{center}
\begin{tikzpicture}[baseline=(2),inner sep=0pt, circle, minimum size=4pt, scale =.8]
\node (root) at (0,0.5) [fill=none] {};
\node (0) at (0,0) [circle, fill] {};
\node at (0,0) [circle, draw, minimum size= 6pt]{};
\node (1) at (0,-1)[circle, fill] {};
\node (f1) at (-1,-2) [circle, fill]{};
\node (2) at (1,-2) [circle, fill]{};
\node (f2) at (2,-3) [circle, fill]{};
\node at (2,-3) [circle, draw, minimum size= 6pt]{};
\node at (2,-3) [circle, draw, minimum size= 8pt]{};
\node at (2,-3) [circle, draw, minimum size= 10pt]{};
\node (3) at (0,-3) [circle, fill]{};
\node (f3) at (-1,-4) [circle, fill]{};
\node (f4) at (1,-4) [circle, fill]{};
\draw (0) [ultra thick, red] to (root);
\draw (1) to (0);
\draw (1) to (2);
\draw (1) to (f1);
\draw (f2) [ultra thick, red]to (2);
\draw (3) to (2);
\draw (3) to (f3);
\draw (3) to (f4);
\draw (f1) [bend left, dashed, <-]to (0);
\draw (f2) [dashed, bend right, <-]to (0);
\draw (f3) [bend left, dashed, <-]to (0);
\draw (f4) [dashed, bend right, <-]to (0);
\end{tikzpicture}
$\longrightarrow$
\begin{tikzpicture}[baseline=(1),inner sep=0pt, circle, minimum size=4pt, scale =.8]
\node (unaire) at (-1,2) [circle, fill, blue] {};
\node (add1) at (-1,1) [circle, fill] {};
\node (ff1) at (-2,0) [circle, fill, blue]{};
\node (0) at (0,0) [circle, fill] {};
\node (1) at (0,-1)[circle, fill] {};
\node (f1) at (-1,-2) [circle, fill]{};
\node (2) at (1,-2) [circle, fill]{};
\node (add2) at (2,-3) [circle, fill]{};
\node (add3) at (3,-4) [circle, fill]{};
\node (add4) at (4,-5) [circle, fill]{};
\node (ff2) at (1.5,-4) [circle, fill, blue]{};
\node (ff3) at (5,-6) [circle, fill, blue]{};
\node (ff4) at (3,-6) [circle, fill, blue]{};
\node (f2) at (2,-5) [circle, fill]{};
\node (3) at (0,-3) [circle, fill]{};
\node (f3) at (-1,-4) [circle, fill]{};
\node (f4) at (0.5,-4) [circle, fill]{};
%
\draw (unaire) [ultra thick, red]to (add1);
\draw (add1) to (ff1);
\draw (add1) [ultra thick, red] to (0);
\draw (1) to (0);
\draw (1) to (2);
\draw (1) to (f1);
\draw (add2) [ultra thick, red]to (2);
\draw (add2) [ultra thick, red]to (add3);
\draw (f2) [ultra thick, red]to (add3);
\draw (ff2) to (add2);
\draw (add4) to (add3);
\draw (add4) to (ff4);
\draw (ff3) to (add4);
\draw (3) to (2);
\draw (3) to (f3);
\draw (3) to (f4);
\draw (f1) [bend left, dashed, <-]to (0);
\draw (f2) [dashed, bend right, <-]to (0);
\draw (f3) [bend left, dashed, <-]to (0);
\draw (f4) [dashed, bend right, <-]to (0);
\draw (ff1) [dashed, bend left, <-, blue]to (unaire);
\draw (ff2) [dashed, bend right, <-, blue]to (unaire);
\draw (ff3) [dashed, bend right, <-, blue]to (unaire);
\draw (ff4) [dashed, bend right, <-, blue]to (unaire);
\end{tikzpicture}
\caption{\small To the left, a BCI(4) term with a node pointed once and another pointed 3 times where
pointing at a node is represented by encircling the dot representing it in the figure. So the root
on top is pointed at once, the right-most leaf three times. The
corresponding hit edges are the thick ones. 
\newline
\hspace*{2ex} At the right, a possible BCI(4) obtained from the left term.
Each thick edge has been replaced by a thick path where binary trees have been attached; their
leaves are linked to the newly created unary node at the root. The root edge (on top of the left
term) has been replaced by a path of length two (having thus three nodes) and a size one tree has
been attached to the left at the middle node; the second thick edge of the left term has been
replaced by a path of length 3 with two attachments: a size one tree left from the second node and
a size 3 tree to the right of the third node of the path. } \label{figbci}
\end{center}
\end{figure}

Now let us count in how many ways this can be done. Each edge which is hit $i$ times is actually
replaced by a sequence of left or right binary trees. The generating function associated to binary
trees is 
$T(u)=\sum_{n\ge 1} C_{n-1}u^n=(1-\sqrt{1-4u}\,)/2$. Thus the number of such sequences with exactly
$i$ leaves is 
$$
[u^i]\frac1{1-2T(u)}=[u^i]\frac1{\sqrt{1-4u}}=\binom{2i}i.
$$
Note that $s_i$ of the $l$ edges are hit $i$ times, $i=1,\dots,p$. The number of ways to partition
the $l$ edges w.r.t. the multiplicity of the hits is $\binom{l}{s_1, \dots, s_p}$. Then each of 
the $s_i$ edges which is hit $i$ times is replaced by one of the $\binom{2i}i$ possible sequences
of binary trees. Therefore there are $\prod_{i=1}^p \binom{2i}i^{s_i}$ ways of doing the whole 
replacement. Finally, note that choosing 
$l$ distinct edges corresponds to applying the operator $z^l D^l/l!$ on the level of generating
functions and the $2p+1$ new nodes created during the replacement process yield a factor
$z^{2p+1}$. 
\epf

\bp\label{FPS}
Let $F(z)$ denote a formal power series (with real coefficients), 
$D_u=\partial/\partial u$, the formal derivative, and $U$ the operator $G(u)\mapsto G(0)$, 
$G(u)$ being a formal power series. Then 
$$
\Delta_p F(z)=\frac{z^{2p+1}}{p!} UD_u^p
F\(\frac{z}{\sqrt{1-4u}}\)=z^{2p+1}[u^p]F\(\frac{z}{\sqrt{1-4u}}\).
$$
\ep

\bpf
The second equation is obvious since $UD_u^p/p!=[u^p]$ is exactly Taylor's theorem. 
For proving the first equation set $D_z=\partial/\partial z$ and 
$
f(u):=1/\sqrt{1-4u}= \sum_{i\ge 0}\binom{2i}i u^i.
$ 
Therefore by Fa\`a di Bruno's formula (see e.g. \cite[p.~137]{Comtet74}) we obtain 
\begin{align*} 
\frac{z^{2p+1}}{p!}UD_u^p F(zf(u))&= 
\frac{z^{2p+1}}{p!}
\sum_{\sum_{i=1}^p is_i=p} \frac{p!}{s_1!\cdots s_p!}
(D^{s_1+\cdots+s_p}F)(zf(0)) \prod_{m=1}^p\(\frac1{m!} UD_u^m(zf(u))\)^{s_m} 
\\
&=z^{2p+1}\sum_{\sum_{i=1}^p is_i=p} \frac1{s_1!\cdots s_p!}
(D^{s_1+\cdots+s_p}F)(zf(0)) \prod_{m=1}^p \(z\binom{2m}m\)^{s_m} \\
&=z^{2p+1}\sum_{l=1}^p \frac1{l!}\binom{l}{s_1, \dots, s_p} \prod_{m=1}^p \binom{2m}m^{s_m} 
z^l D^l F(z) 
\\
&=\sum_{l=1}^p \frac{\alpha_{l,p}}{l!} z^{l+2p+1} D^l F(z) = \Delta_p F(z).
\end{align*} 
where we substituted $s_1+\cdots+s_p=l$ in the second line and split the sum according to the
value of $l$ and used $f(0)=1$ in the third line. 
\epf

\brem
More heuristically, we could argue in the following way: 
Regard $F(z)$ as a generating function of a tree-like structure where $z$ marks the number of
nodes. 
Then $F(z/\sqrt{1-4u}\,)$ is the generating function where the nodes are substituted by a
node and a sequence of ``left-or-right'' binary trees where the number of leaves is marked by $u$.
Thus $[u^p]F(z/\sqrt{1-4u}\,)$ is the generating function of those objects where the binary trees
introduced by the substitution altogether contain exactly $p$ leaves. The term $z^{2p+1}$ accounts
for introducing the $2p+1$ new nodes. This comes from counting the nodes of the binary trees coming from
the substitution, adding an extra root for each of these trees and adding a new root to the total
structure. This is precisely what $\Delta_p$ does. 
\erem



The derivation of the differential equation of the generating function for $BCK(p)$-terms is a
little more involved. Note that the differential operator
$\Delta_p$ corresponds to $p$ pointers from the root to some leaves. One is tempted
to replace $\Delta_p$ in \eqref{bci} by a sum of $\Delta_l$'s to take into account less than $p$
pointers. But this is not entirely correct. 

\bp\label{BCK}
Let $F_p(z)$ be the generating function associated to $BCK(p)$-terms. Then $F_p(z)=Y(z/(1-z))$ where $Y(z)$
is the unique power series $Y(z)=\sum_{n\ge 0} Y_nz^n$ with nonnegative coefficients which
satisfies 
\begin{equation} \label{Y_eq}
Y(z) = \sum_{l=1}^p C_{l-1} z^{2l} + z Y(z)^2
+ \left( \sum_{l=1}^p \Delta_l \right) Y(z).
\end{equation} 
\ep

\bpf
The (in some sense) minimal $BCK(p)$-terms are binary trees with at most $p$ leaves and a unary
root node pointing at all the leaves. This gives the first term on the right-hand side of
\eqref{Y_eq}. 

Note that a unary node may also have zero pointers. A unary node with zero pointers which is not on top
of the tree cannot be generated directly by a specification similar to \eqref{speci}. Therefore we
first construct terms where each unary node has at least one pointer. Similar arguments as in the
$BCI$ case then lead directly to \eqref{Y_eq}. Finally, replace the edges by paths which exactly
corresponds to the substitution $z\to z/(1-z)$.
\epf

An alternative approach is to start with Motzkin trees with an additional root having pointers to
all leaves as minimal structures. The terms with a
unary root node can then be generated in the following way: Fix the number $l$ 
of pointers we want to have at the root and then do an edge hitting process as in the $BCI$ case.
But instead of substituting the hit edges by sequences of left-or-right binary trees, use
sequences of left-or-right Motzkin trees with an additional unary root node (corresponding to the
nodes in the paths which substitute the hit edges) such that these trees have altogether $l$
leaves. Recalling that on the level of generating functions edge hitting corresponds to applying a
differential operator, we get in that way a differential equation for $F_p(z)$. 

\bp
Let $M(z,u)$ denote the generating function of Motzkin trees where $z$ marks the size (i.e. the
total number of nodes) and $u$ marks the number of leaves. This function is given by the unique
power series solution of $M(z,u)=uz+zM(z,u)+zM(z,u)^2$, that is 
\begin{equation} \label{motz}
M(z,u)=\frac{1-z-\sqrt{(1-z)^2-4uz^2}}{2z}.
\end{equation} 
Then $F_p(z)$ is given as the solution of 
\begin{equation}\label{bck}
F_p(z)=z[u^p]\frac{M(z,u)}{1-u} + z F_p(z)^2 + z 
[u^p]\frac1{1-u}\,F_p\(\frac{z}{1-2zM(z,u)}\).
\end{equation}
\ep

\bpf
This is a direct consequence of the remarks above and Proposition~\ref{FPS}.
\epf

Let $\lambda_n$ denote the number of closed lambda-terms and $\Lambda(z)=\sum_{n\ge 1} \lambda_n
z^n$. Then we can use the two approaches presented above to find functional equations for
$\Lambda(z)$. 

\bp
Let $C(z)=(1-\sqrt{1-4z^2}\,)/2$ be the generating function associated to binary trees with an
extra unary
root node and counted by the number of nodes. Furthermore, let $\tilde \Lambda(z)$ be the power 
series solution of 
\begin{equation} \label{indirect}
\tilde \Lambda (z) =
C(z) + z \tilde \Lambda(z)^2 + z \tilde \Lambda \left( \frac{z}{1-2 C(z)} \right) - z \tilde
\Lambda(z).
\end{equation} 
Then $\Lambda(z) = \tilde \Lambda (z/(1-z))$. Moreover, we have
\begin{equation} \label{direct}
\Lambda(z)=zM(z,1)+z\Lambda(z)^2+z\Lambda\(\frac{z}{1-2zM(z,1)}\).
\end{equation} 
\ep

\bpf
To prove \eqref{indirect} we can proceed as in the proofs of Propositions~\ref{BCI} and~\ref{BCK}
but allowing an unbounded number of edge hits instead. Thus, if $\tilde \Lambda(z)$ is the
generating function associated to closed lambda-terms where each unary node carries at least one
pointer, then   
\[
\tilde\Lambda (z) = \sum_{p\geq 1} C_{p-1} z^{2p} + z \tilde\Lambda(z)^2 + \dd \tilde\Lambda (z)
\]
where $\dd = \sum_{p\ge 1} \Delta_p$. Now applying Proposition~\ref{FPS} yields \eqref{indirect}. As
in the BCK case, in order to create unary nodes carrying no pointers we replace the edges by paths
which yields $\Lambda(z)= \tilde \Lambda (z/(1-z))$ and completes the proof of \eqref{indirect}.

\begin{figure}[ht]
\begin{center}
%
%
%
\begin{tikzpicture}[baseline=(4),inner sep=0pt, circle, minimum size=6pt, scale =0.6]
\node (1) at (6,6)[rectangle, fill] {};
\node (2) at (6,5) [rectangle, fill]{};
\node (3) at (5,4) [rectangle, fill]{};
\node (4) at (7,4) [rectangle, fill]{};
\node (5) at (7,3) [rectangle, fill]{};
\node (8) at (7,2) [rectangle, fill]{};
\node (6) at (6,1)[rectangle, fill] {};
\node (7) at (8,1)[rectangle, fill]
 {
};

\draw (1) to[very thick] (2);
\draw (3) to[very thick] (2);
\draw (4) to[very thick] (2);
\draw (4) to[very thick] (5);
\draw (6) to[very thick] (8);
\draw (8) to[very thick] (7);
\draw (8) to[very thick] (5);

\draw (3) to[bend left, dashed, <-] (1);
\draw (6) to[dashed, bend left, <-] (1);
\draw (7) to[dashed, bend right, <-] (4);
\end{tikzpicture}
$\longrightarrow$
\begin{tikzpicture}[baseline=(4),inner sep=0pt, circle, minimum size=4pt, scale =0.6]
\node (1) at (6,6)[shape=rectangle, draw,inner sep=1pt] {
\begin{tikzpicture}[baseline=(a),inner sep=0pt, circle, minimum size=3pt, scale =0.2]
\node (a) at (1,1) [circle, fill]{};
\node (b) at (0,0)[circle, fill] {};
\draw (a) to (b);
\end{tikzpicture}
, 
\begin{tikzpicture}[baseline=(a),inner sep=0pt, circle, minimum size=3pt, scale =0.2]
\node (a) at (1,1) [circle, fill]{};
\node (b) at (2,0)[circle, fill] {};
\node (c) at (3, -1) [circle, fill] {};
\node (d) at (1, -1) [circle, fill] {};
\node (e) at (4, -3) [circle, fill] {};
\node (f) at (2, -3) [circle, fill] {};
\node (u) at (3, -2)[circle, fill] {};
\draw (a) to (b);
\draw (c) to (b);
\draw (d) to (b);
\draw (u) to (e);
\draw (u) to (f);
\draw (u) to (c);
\end{tikzpicture}
,
\begin{tikzpicture}[baseline=(a),inner sep=0pt, circle, minimum size=3pt, scale =0.2]
\node (a) at (1,1) [circle, fill]{};
\end{tikzpicture}
};
\node (2) at (6,5) [rectangle, draw,inner sep=1pt]{
\begin{tikzpicture}[baseline=(a),inner sep=0pt, circle, minimum size=3pt, scale =0.2]
\node (a) at (1,1) [circle, fill]{};
\end{tikzpicture}};
\node (3) at (5,4) [rectangle, draw,inner sep=1pt]{
 \begin{tikzpicture}[baseline=(a),inner sep=0pt, circle, minimum size=3pt, scale =0.2]
 \node (a) at (1,1) [circle, fill]{};
 \end{tikzpicture}
};
\node (4) at (7,4) [rectangle, draw,inner sep=1pt]{
 \begin{tikzpicture}[baseline=(a),inner sep=0pt, circle, minimum size=3pt, scale =0.2]
 \node (a) at (1,1) [circle, fill]{};
 \node (b) at (0,0)[circle, fill] {};
 \node (c) at (0, -1)[circle, fill] {};
 \draw (a) to (b); \draw (c) to (b);
 \end{tikzpicture}
 ,
 \begin{tikzpicture}[baseline=(a),inner sep=0pt, circle, minimum size=3pt, scale =0.2]
 \node (a) at (1,1) [circle, fill]{};
 \end{tikzpicture}
};
\node (5) at (7,3) [rectangle, draw,inner sep=1pt]{
 \begin{tikzpicture}[baseline=(a),inner sep=0pt, circle, minimum size=3pt, scale =0.2]
 \node (a) at (1,1) [circle, fill]{};
 \end{tikzpicture}
};
\node (8) at (7,2)[rectangle, draw,inner sep=1pt]{
 \begin{tikzpicture}[baseline=(a),inner sep=0pt, circle, minimum size=3pt, scale =0.2]
 \node (a) at (1,1) [circle, fill]{};
 \end{tikzpicture}
};
\node (6) at (6,1)[rectangle, draw,inner sep=1pt]{
 \begin{tikzpicture}[baseline=(a),inner sep=0pt, circle, minimum size=3pt, scale =0.2]
 \node (a) at (1,1) [circle, fill]{};
 \end{tikzpicture}
};
\node (7) at (8,1)[rectangle, draw,inner sep=1pt]{
 \begin{tikzpicture}[baseline=(a),inner sep=0pt, circle, minimum size=3pt, scale =0.2]
 \node (a) at (1,1) [circle, fill]{};
 \node (b) at (2,0)[circle, fill] {};
 \draw (a) to (b); 
 \end{tikzpicture}
 ,
 \begin{tikzpicture}[baseline=(a),inner sep=0pt, circle, minimum size=3pt, scale =0.2]
 \node (a) at (1,1) [circle, fill]{};
 \node (b) at (2,0)[circle, fill] {};
 \node (d) at (2,-1)[circle, fill] {};
 \draw (a) to (b); \draw (b) to (d); 
 \end{tikzpicture}
 ,
 \begin{tikzpicture}[baseline=(a),inner sep=0pt, circle, minimum size=3pt, scale =0.2]
 \node (a) at (1,1) [circle, fill]{};
 \end{tikzpicture}
};

\draw (node cs: name=1, anchor=south) to[very thick] (2);
\draw (3) to[very thick] (2);
\draw (4) to[very thick] (2);
\draw (4) to[very thick] (5);
\draw (6) to[very thick] (8);
\draw (8) to[very thick] (7);
\draw (8) to[very thick] (5);

\draw (3) to[bend left, dashed, <-] (1
);
\draw (6) to[dashed, bend left, <-] (1
);
\draw (7) to[dashed, bend right, <-] (4);
\end{tikzpicture}
$\longrightarrow$
\begin{tikzpicture}[baseline=(v3),inner sep=0pt, circle, minimum size=4pt, scale =0.6]
\node (0) at (0,0) [circle, fill] {};
\node (1) at (0,-1)[circle, fill] {};
\node (f1) at (-1,-2) [circle, fill]{};
\node (2) at (1,-2) [circle, fill]{};
\node (3) at (3,-3) [circle, fill]{};
\node (f2) at (2,-4) [circle, fill]{};
\node (4) at (4,-4) [circle, fill]{};
\node (5) at (4,-5)[circle, fill] {};
\node (f3) at (3,-6)[circle, fill] {};
\node (f4) at (5,-6)[circle, fill] {};
\node (v1) at (0,-3) [rectangle, fill,minimum size=6pt]{};
\node (v2) at (0,-4) [rectangle, fill,minimum size=6pt]{};
\node (f1v1) at (-1,-5)[rectangle, fill,minimum size=6pt] {};
\node (6) at (1,-5)[circle, fill] {};
\node (7) at (0,-6)[circle, fill] {};
\node (f5) at (0,-7) [circle, fill]{};
\node (v3) at (2,-6) [rectangle, fill,minimum size=6pt]{};
\node (v4) at (2,-7)[rectangle, fill,minimum size=6pt] {};
\node (v5) at (2,-8)[rectangle, fill,minimum size=6pt] {};
\node (f2v1) at (1,-9)[rectangle, fill,minimum size=6pt] {};
\node (8) at (3,-9)[circle, fill] {};
\node (f6) at (4,-10)[circle, fill] {};
\node (9) at (2,-10)[circle, fill] {};
\node (f1v3) at (1,-11)[rectangle, fill,minimum size=6pt] {};
\node (10) at (3,-11)[circle, fill] {};
\node (f7) at (3,-12)[circle, fill] {};

\draw (0) to[very thick] (1);
\draw (1) to [very thick](2);
\draw (3) to (2);
\draw (4) to (3);
\draw (4) to (5);
\draw (f1) to (1);
\draw (v1) to[very thick] (2);
\draw (3) to (f2);
\draw (5) to (f3);
\draw (f4) to (5);
\draw (v1) to[very thick] (v2);
\draw (f1v1) to[very thick] (v2);
\draw (6) to[very thick] (v2);
\draw (6) to (7);
\draw (7) to (f5);
\draw (6) to[very thick] (v3);
\draw (v3) to[very thick] (v4);
\draw (v4) to[very thick] (v5);
\draw (v5) to[very thick] (f2v1);
\draw (v5) to[very thick] (8);
\draw (8) to (f6);
\draw (9) to[very thick] (8);
\draw (9) to[very thick] (f1v3);
\draw (9) to (10);
\draw (10) to (f7);
\draw (f1v1) to[bend left, dashed, <-] (v1);
\draw (f2v1) to[dashed, bend left, <-] (v1);
\draw (f1v3) to[dashed, bend right, <-] (v3);
\draw (f1) to[bend left, dashed, <-] (0);
\draw (f2) to[dashed, bend right, <-] (0);
\draw (f3) to[dashed, bend right, <-] (0);
\draw (f4) to[dashed, bend right, <-] (0);
\draw (f5) to[bend right, dashed, <-] (0);
\draw (f6) to[dashed, bend right, <-] (0);
\draw (f7) to[dashed, bend right, <-] (0);
\end{tikzpicture}
\caption{\small A step of the grafting expansion of a lambda-term}
\label{graftingfigure}
\end{center}
\end{figure}
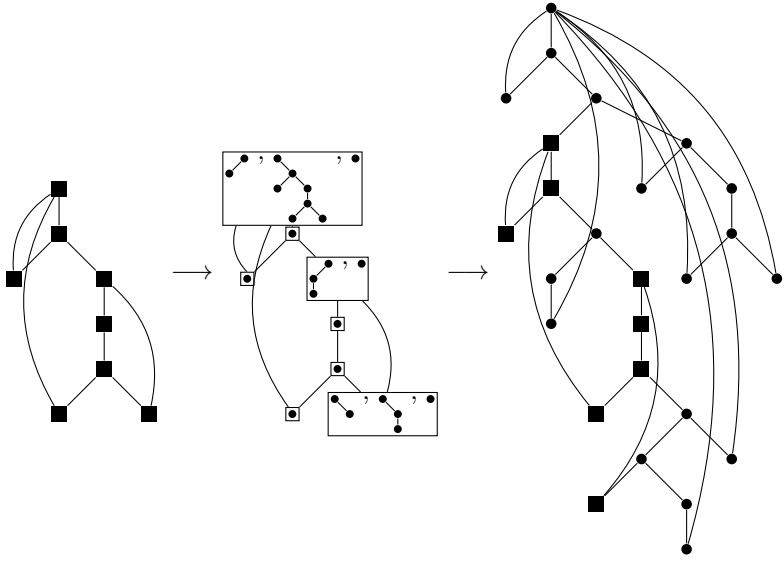

Alternatively, the lambda-terms with a unary root node can be created by starting with Motzkin
trees with a unary node on top pointing to all leaves. These initial configurations are then
expanded iteratively by substituting the edges by paths and attaching nodes, either left or right,
which are (unary) roots of Motzkin trees, each binding all the leaves of its subtree. For an
illustration of the expansion process, Figure~\ref{graftingfigure} shows one step in this
expansion process (not the initial one). Figure~\ref{graftinginverse} presents one step of the
reverse process. 
\epf

\begin{figure}[ht]
\begin{center}
\begin{tikzpicture}[baseline=(v3),inner sep=0pt, circle, minimum size=4pt, scale =0.4]
\node (0) at (0,0) [circle, fill] {};
\node (1) at (0,-1)[circle, fill] {};
\node (f1) at (-1,-2) [circle, fill]{};
\node (2) at (1,-2) [circle, fill]{};
\node (3) at (3,-3) [circle, fill]{};
\node (f2) at (2,-4) [circle, fill]{};
\node (4) at (4,-4) [circle, fill]{};
\node (5) at (4,-5)[circle, fill] {};
\node (f3) at (3,-6)[circle, fill] {};
\node (f4) at (5,-6)[circle, fill] {};
\node (v1) at (0,-3) [circle, fill,minimum size=4pt]{};
\node (v2) at (0,-4) [circle, fill,minimum size=4pt]{};
\node (f1v1) at (-1,-5)[circle, fill,minimum size=4pt] {};
\node (6) at (1,-5)[circle, fill] {};
\node (7) at (0,-6)[circle, fill] {};
\node (f5) at (0,-7) [circle, fill]{};
\node (v3) at (2,-6) [circle, fill,minimum size=4pt]{};
\node (v4) at (2,-7)[circle, fill,minimum size=4pt] {};
\node (v5) at (2,-8)[circle, fill,minimum size=4pt] {};
\node (f2v1) at (1,-9)[circle, fill,minimum size=4pt] {};
\node (8) at (3,-9)[circle, fill] {};
\node (f6) at (4,-10)[circle, fill] {};
\node (9) at (2,-10)[circle, fill] {};
\node (f1v3) at (1,-11)[circle, fill,minimum size=4pt] {};
\node (10) at (3,-11)[circle, fill] {};
\node (f7) at (3,-12)[circle, fill] {};

\draw (0) to[very thick] (1);
\draw (1) to [very thick](2);
\draw (3) to[very thick] (2);
\draw (4) to[very thick] (3);
\draw (4) to[very thick] (5);
\draw (f1) to[very thick] (1);
\draw (v1) to[very thick] (2);
\draw (3) to[very thick] (f2);
\draw (5) to[very thick] (f3);
\draw (f4) to[very thick] (5);
\draw (v1) to[very thick] (v2);
\draw (f1v1) to[very thick] (v2);
\draw (6) to[very thick] (v2);
\draw (6) to[very thick] (7);
\draw (7) to[very thick] (f5);
\draw (6) to[very thick] (v3);
\draw (v3) to[very thick] (v4);
\draw (v4) to[very thick] (v5);
\draw (v5) to[very thick] (f2v1);
\draw (v5) to[very thick] (8);
\draw (8) to[very thick] (f6);
\draw (9) to[very thick] (8);
\draw (9) to[very thick] (f1v3);
\draw (9) to [very thick](10);
\draw (10) to[very thick] (f7);
\draw (f1v1) to[bend left, dashed, <-] (v1);
\draw (f2v1) to[dashed, bend left, <-] (v1);
\draw (f1v3) to[dashed, bend right, <-] (v3);
\draw (f1) to[bend left, dashed, <-] (0);
\draw (f2) to[dashed, bend right, <-] (0);
\draw (f3) to[dashed, bend right, <-] (0);
\draw (f4) to[dashed, bend right, <-] (0);
\draw (f5) to[bend right, dashed, <-] (0);
\draw (f6) to[dashed, bend right, <-] (0);
\draw (f7) to[dashed, bend right, <-] (0);
\end{tikzpicture}
$\longrightarrow$
\begin{tikzpicture}[baseline=(v3),inner sep=0pt, circle, minimum size=4pt, scale =0.4]
\node (0) at (0,0) [draw] {};
\node (1) at (0,-1)[circle, fill] {};
\node (f1) at (-1,-2) [draw]{};
\node (2) at (1,-2) [circle, fill]{};
\node (3) at (3,-3) [circle, fill]{};
\node (f2) at (2,-4) [draw]{};
\node (4) at (4,-4) [circle, fill]{};
\node (5) at (4,-5)[circle, fill] {};
\node (f3) at (3,-6)[draw] {};
\node (f4) at (5,-6)[draw] {};
\node (v1) at (0,-3) [circle, fill,minimum size=4pt]{};
\node (v2) at (0,-4) [circle, fill,minimum size=4pt]{};
\node (f1v1) at (-1,-5)[circle, fill,minimum size=4pt] {};
\node (6) at (1,-5)[circle, fill] {};
\node (7) at (0,-6)[circle, fill] {};
\node (f5) at (0,-7) [draw]{};
\node (v3) at (2,-6) [circle, fill,minimum size=4pt]{};
\node (v4) at (2,-7)[circle, fill,minimum size=4pt] {};
\node (v5) at (2,-8)[circle, fill,minimum size=4pt] {};
\node (f2v1) at (1,-9)[circle, fill,minimum size=4pt] {};
\node (8) at (3,-9)[circle, fill] {};
\node (f6) at (4,-10)[draw] {};
\node (9) at (2,-10)[circle, fill] {};
\node (f1v3) at (1,-11)[circle, fill,minimum size=4pt] {};
\node (10) at (3,-11)[circle, fill] {};
\node (f7) at (3,-12)[draw] {};

\draw (0) to(1);
\draw (1) to[very thick] (2);
\draw (3) to[very thick] (2);
\draw (4) to[very thick] (3);
\draw (4) to [very thick](5);
\draw (f1) to (1);
\draw (v1) to[very thick] (2);
\draw (3) to (f2);
\draw (5) to(f3);
\draw (f4) to (5);
\draw (v1) to[very thick] (v2);
\draw (f1v1) to[very thick](v2);
\draw (6) to[very thick] (v2);
\draw (6) to[very thick] (7);
\draw (7) to (f5);
\draw (6) to[very thick] (v3);
\draw (v3) to[very thick] (v4);
\draw (v4) to[very thick] (v5);
\draw (v5) to[very thick] (f2v1);
\draw (v5) to[very thick] (8);
\draw (8) to(f6);
\draw (9) to[very thick] (8);
\draw (9) to[very thick] (f1v3);
\draw (9) to[very thick] (10);
\draw (10) to (f7);
\draw (f1v1) to[bend left, dashed, <-] (v1);
\draw (f2v1) to[dashed, bend left, <-] (v1);
\draw (f1v3) to[dashed, bend right, <-] (v3);
\end{tikzpicture}
$\longrightarrow$
\begin{tikzpicture}[baseline=(v3),inner sep=0pt, circle, minimum size=4pt, scale =0.4]
\node (0) at (0,0) [] {};
\node (1) at (0,-1)[draw] {};
\node (f1) at (-1,-2) []{};
\node (2) at (1,-2) [circle, fill]{};
\node (3) at (3,-3) [draw]{};
\node (f2) at (2,-4) []{};
\node (4) at (4,-4) [circle, fill]{};
\node (5) at (4,-5)[draw] {};
\node (f3) at (3,-6)[] {};
\node (f4) at (5,-6)[] {};
\node (v1) at (0,-3) [circle, fill,minimum size=4pt]{};
\node (v2) at (0,-4) [circle, fill,minimum size=4pt]{};
\node (f1v1) at (-1,-5)[circle, fill,minimum size=4pt] {};
\node (6) at (1,-5)[circle, fill] {};
\node (7) at (0,-6)[draw] {};
\node (f5) at (0,-7) []{};
\node (v3) at (2,-6) [circle, fill,minimum size=4pt]{};
\node (v4) at (2,-7)[circle, fill,minimum size=4pt] {};
\node (v5) at (2,-8)[circle, fill,minimum size=4pt] {};
\node (f2v1) at (1,-9)[circle, fill,minimum size=4pt] {};
\node (8) at (3,-9)[draw] {};
\node (f6) at (4,-10)[] {};
\node (9) at (2,-10)[circle, fill] {};
\node (f1v3) at (1,-11)[circle, fill,minimum size=4pt] {};
\node (10) at (3,-11)[draw] {};
\node (f7) at (3,-12)[] {};

\draw (1) to[] (2);
\draw (3) to[very thick] (2);
\draw (4) to[very thick] (3);
\draw (4) to [](5);
\draw (v1) to[very thick] (2);
\draw (v1) to[very thick] (v2);
\draw (f1v1) to[very thick](v2);
\draw (6) to[very thick] (v2);
\draw (6) to[] (7);
\draw (6) to[very thick] (v3);
\draw (v3) to[very thick] (v4);
\draw (v4) to[very thick] (v5);
\draw (v5) to[very thick] (f2v1);
\draw (v5) to[very thick] (8);
\draw (9) to[very thick] (8);
\draw (9) to[very thick] (f1v3);
\draw (9) to[] (10);
\draw (f1v1) to[bend left, dashed, <-] (v1);
\draw (f2v1) to[dashed, bend left, <-] (v1);
\draw (f1v3) to[dashed, bend right, <-] (v3);
\end{tikzpicture}
$\longrightarrow$
\begin{tikzpicture}[baseline=(v3),inner sep=0pt, circle, minimum size=4pt, scale =0.4]
\node (2) at (1,-2) [circle, fill]{};
\node (3) at (3,-3) [minimum size=0pt]{};
\node (4) at (4,-4) [draw]{};
\node (v1) at (0,-3) [circle, fill,minimum size=4pt]{};
\node (v2) at (0,-4) [circle, fill,minimum size=4pt]{};
\node (f1v1) at (-1,-5)[circle, fill,minimum size=4pt] {};
\node (6) at (1,-5)[draw] {};
\node (v3) at (2,-6) [circle, fill,minimum size=4pt]{};
\node (v4) at (2,-7)[circle, fill,minimum size=4pt] {};
\node (v5) at (2,-8)[circle, fill,minimum size=4pt] {};
\node (f2v1) at (1,-9)[circle, fill,minimum size=4pt] {};
\node (8) at (3,-9)[minimum size=0pt] 
{};
\node (9) at (2,-10)[draw] {};
\node (f1v3) at (1,-11)[circle, fill,minimum size=4pt] {};

\draw (3) to[] (2);
\draw (4) to[] (3);
\draw (v1) to[very thick] (2);
\draw (v1) to[very thick] (v2);
\draw (f1v1) to[very thick](v2);
\draw (6) to[very thick] (v2);
\draw (6) to[very thick] (v3);
\draw (v3) to[very thick] (v4);
\draw (v4) to[very thick] (v5);
\draw (v5) to[very thick] (f2v1);
\draw (v5) to[very thick] (8);
\draw (9) to[very thick] (8);
\draw (9) to[very thick] (f1v3);
\draw (f1v1) to[bend left, dashed, <-] (v1);
\draw (f2v1) to[dashed, bend left, <-] (v1);
\draw (f1v3) to[dashed, bend right, <-] (v3);
\end{tikzpicture}
$\longrightarrow$
\begin{tikzpicture}[baseline=(v3),inner sep=0pt, circle, minimum size=4pt, scale =0.4]
\node (2) at (1,-2) [draw]{};
\node (3) at (3,-3) [minimum size=0pt]{};
\node (v1) at (0,-3) [circle, fill,minimum size=4pt]{};
\node (v2) at (0,-4) [circle, fill,minimum size=4pt]{};
\node (f1v1) at (-1,-5)[circle, fill,minimum size=4pt] {};
\node (6) at (1,-5)[minimum size=0pt] {};
\node (v3) at (2,-6) [circle, fill,minimum size=4pt]{};
\node (v4) at (2,-7)[circle, fill,minimum size=4pt] {};
\node (v5) at (2,-8)[circle, fill,minimum size=4pt] {};
\node (f2v1) at (1,-9)[circle, fill,minimum size=4pt] {};
\node (8) at (3,-9)[minimum size=0pt] {};
\node (9) at (2,-10)[minimum size=0pt] {};
\node (f1v3) at (1,-11)[circle, fill,minimum size=4pt] {};

\draw (v1) to (2);
\draw (v1) to[very thick] (v2);
\draw (f1v1) to[very thick](v2);
\draw (6) to[very thick] (v2);
\draw (6) to[very thick] (v3);
\draw (v3) to[very thick] (v4);
\draw (v4) to[very thick] (v5);
\draw (v5) to[very thick] (f2v1);
\draw (v5) to[very thick] (8);
\draw (9) to[very thick] (8);
\draw (9) to[very thick] (f1v3);
\draw (f1v1) to[bend left, dashed, <-] (v1);
\draw (f2v1) to[dashed, bend left, <-] (v1);
\draw (f1v3) to[dashed, bend right, <-] (v3);
\end{tikzpicture}
$\longrightarrow$
\begin{tikzpicture}[baseline=(v3),inner sep=0pt, circle, minimum size=4pt, scale =0.4]
\node (v1) at (0,-3) [circle, fill,minimum size=4pt]{};
\node (v2) at (0,-4) [circle, fill,minimum size=4pt]{};
\node (f1v1) at (-1,-5)[circle, fill,minimum size=4pt] {};
\node (6) at (1,-5)[minimum size=0pt] {};
\node (v3) at (2,-6) [circle, fill,minimum size=4pt]{};
\node (v4) at (2,-7)[circle, fill,minimum size=4pt] {};
\node (v5) at (2,-8)[circle, fill,minimum size=4pt] {};
\node (f2v1) at (1,-9)[circle, fill,minimum size=4pt] {};
\node (8) at (3,-9)[minimum size=0pt] {};
\node (9) at (2,-10)[minimum size=0pt] {};
\node (f1v3) at (1,-11)[circle, fill,minimum size=4pt] {};

\draw (v1) to[very thick] (v2);
\draw (f1v1) to[very thick](v2);
\draw (6) to[very thick] (v2);
\draw (6) to[very thick] (v3);
\draw (v3) to[very thick] (v4);
\draw (v4) to[very thick] (v5);
\draw (v5) to[very thick] (f2v1);
\draw (v5) to[very thick] (8);
\draw (9) to[very thick] (8);
\draw (9) to[very thick] (f1v3);
\draw (f1v1) to[bend left, dashed, <-] (v1);
\draw (f2v1) to[dashed, bend left, <-] (v1);
\draw (f1v3) to[dashed, bend right, <-] (v3);
\end{tikzpicture}
$\longrightarrow$
\begin{tikzpicture}[baseline=(4),inner sep=0pt, circle, minimum size=4pt, scale =0.4]
\node (1) at (6,6)[ fill] {};
\node (2) at (6,5) [ fill]{};
\node (3) at (5,4) [ fill]{};
\node (4) at (7,4) [fill]{};
\node (5) at (7,3) [ fill]{};
\node (8) at (7,2) [ fill]{};
\node (6) at (6,1)[fill] {};
\node (7) at (8,1)[fill]
 {
};

\draw (1) to[very thick] (2);
\draw (3) to[very thick] (2);
\draw (4) to[very thick] (2);
\draw (4) to[very thick] (5);
\draw (6) to[very thick] (8);
\draw (8) to[very thick] (7);
\draw (8) to[very thick] (5);

\draw (3) to[bend left, dashed, <-] (1);
\draw (6) to[dashed, bend left, <-] (1);
\draw (7) to[dashed, bend right, <-] (4);
\end{tikzpicture}
%
%
%
%
\caption{\small Finding the original closed lambda-term: first, the unary root node and all
the leaves bound by the root are coloured white. Then delete the white nodes and
colour all their neighbours white. Now continue recursively, where deletion of white unary nodes
is done by removal and gluing the incident edges together. }
\label{graftinginverse}
\end{center}
\end{figure}
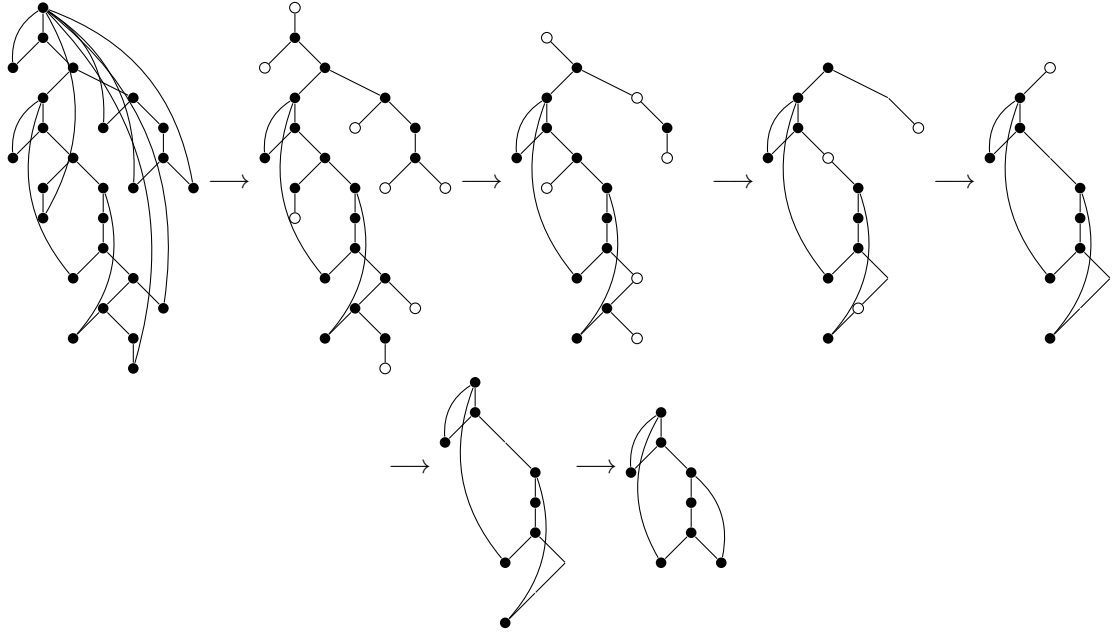

\section{The asymptotic number of $BCI(p)$-terms}\label{bcisec}

Recall that  $G_p(z) = \sum_{n\geq 1} g_{n(2p+1) -1} z^{n(2p+1)-1}$ is the generating function of
the counting sequence of $BCI(p)$ terms. The function $G_p(z)$ satisfies the functional equation
\eqref{bci} which involves the differential operator $\Delta_p$ given by \eqref{defdelta}. Our
goal is now to get a recurrence relation for the coefficients of $G_p(z)$. 

\begin{prop}
\label{prop:recurrence-coefficients}
The coefficients $g_{n(2p+1)-1}$ satisfy the recurrence relation 
\begin{equation}
\label{eq:recurrence-coefficients}
g_{n(2p+1)-1} = \sum_{l=1}^{n-1}
g_{l(2p+1)-1} g_{(n-1-l)(2p+1)-1} +  Q_p(n-1) g_{(n-1)(2p+1)-1}, \mbox{ for } n\ge 2, 
\end{equation}
with initial condition $g_{(2p+1)-1}=C_{p-1}$ and where 
\begin{equation}
\label{eq:defQ(n)}
Q_p(n) 
= \sum_{m=1}^p \alpha_{m,p} \binom{n (2p+1)-1}m.
\end{equation}
\end{prop}

\bpf
Obvious, since the first term on the right-hand side of \eqref{bci} only affects the case $n=1$,
the quadratic term is a Cauchy product and $\Delta_p$ is a linear combination of powers of the
ordinary differential operator which acts on the coefficients of the power series exactly as
shifting and multiplication by $Q_p(n-1)$ do. 
\epf

\bl\label{Q_p}
The polynomials $Q_p(n)$ can be represented more explicitly as 
$$
Q_p(n)=4^p\binom{\(p+\frac12\)n+p-\frac32}{p}.
$$
\el

\bpf
Set $f(u)=1/\sqrt{1-4u}$. It is easy to see that $\alpha_{m,p}=[u^p](f(u)-1)^m$ and that the
coefficient on the right-hand side is zero if $m>p$. Thus we obtain 
\begin{align*} 
Q_p(n)&=\sum_{m=1}^p\binom{(2p+1)n-1}m\alpha_{m,p} \\
&=[u^p]\sum_{m\ge 1} \binom{(2p+1)n-1}m (f(u)-1)^m = [u^p]f(u)^{(2p+1)n-1} \\
&=4^p\binom{\(p+\frac12\)n+p-\frac32}{p}
\end{align*} 
and we are done.
\epf

The key to the asymptotic analysis is a linearization of the differential equation which is
possible due to the fast growth of the coefficients of $G_p(z)$. We start with an auxiliary result
for fast growing sequences saying that in the Cauchy product only the extremal terms are
asymptotically relevant: 

\begin{lemma}
\label{lemma:extremal-terms}
Let $n_0\in\mathbb N$ and $A(z)=\sum_{n\geq n_0} a_n z^n$ be a power series with positive
coefficients (from index $n_0$ on). Assume that there exists $\sigma \geq 1$ 
with $a_{n+1}/a_n = \Omega( n^\sigma )$ as $n\to\infty$. Then
$[z^n] A(z)^2 = 2 a_{n_0} a_{n-n_0} (1 + O(n^{-\sigma}))$ as $n\to\infty$.
If we want the second order term, we take the next two terms, and so on.
\end{lemma}

\begin{proof}
Define $q(n) = a_{n+1}/a_n$; then $1/q(n) = O(n^{-\sigma})$. W.l.o.g. assume that $n$ is odd. Then the
coefficient of $z^n$ in $A(z)^2$ is
\begin{align*}
\sum_{l=n_0}^{n-n_0} a_l a_{n-l} 
&= 
2 a_{n_0} a_{n-n_0} + 2 \sum_{l=1}^{\lfloor n/2\rfloor-n_0} a_{{n_0}+j} a_{n-{n_0}-j} 
\\ &=
 2 a_{n_0} a_{n-{n_0}} \left( 1 + \sum_{l=1}^{\lfloor n/2\rfloor -n_0} \frac{q({n_0})
q({n_0}+1)\cdots q({n_0}+l-1)}{q(n-{n_0}-1)q(n-{n_0}-2)...q(n-{n_0}-l)}
 \right).
\end{align*}
In the case where $n$ is even we have to subtract $\mathbf 1_{ \{n/2\in\mathbb N\} } a_{n/2}^2$ on the r.-h. side. 

The first term of the sum in the last line is $q({n_0})/q(n-{n_0}-1) = O((n-{n_0}-1)^{-\sigma}) = O(n^{-\sigma})$ (recall that ${n_0}$ is a constant). The further terms are of order $O(n^{-2 \sigma})$ and there are not more than $\lfloor n/2\rfloor$ of them. Thus the sum is of order $O(n^{1-2\sigma})=O(n^{-\sigma})$. Hence
\[
[z^n] A^2(z) = 2 a_{n_0} a_{n-{n_0}} (1+O(n^{-\sigma})) \sim [z^n] 2 a_{n_0} z^{n_0} A(z).
\qedhere
\]
\end{proof}

We are now ready to derive bounds for the coefficients of $G_p(z)$. 
\begin{lemma}
\label{lemme:maj-coefficients}
Define $\phi_n = g_{n(2p+1)-1}$, $(n \geq 1)$.
Then we have $\phi_{n+1}/\phi_n = \Omega(n^p)$ as $n\to\infty$.
\end{lemma}
\bpf
By \eqref{eq:recurrence-coefficients} we have $\phi_1=C_{p-1}$ and, for $n\ge 2$,  
\begin{equation}
\label{eq:recurrence-phi}
\phi_n = \sum_{l=1}^{n-1} \phi_l \phi_{n-1-l} + Q_p(n-1) \phi_{n-1} .
\end{equation}
Thus $\phi_n \geq  Q_p(n-1) \, \phi_{n-1}$.
By Lemma~\ref{Q_p} it is obvious that $Q_p(n)$ is a polynomial in $n$ with leading term
$\frac{2^p(2p+1)^p}{p!}n^p$ which implies the result.
\end{proof}

\begin{coroll}
\label{lemma:equiv-sum}
For $p \geq 1$, the sum $\sum_{m+l=n-1} \phi_m \phi_l $ is asymptotically equal to $2 \phi_1 \phi_{n-1} (1+O(1/n^p))$.
\end{coroll}

\brem\label{rem:3}
The intuition behind the considerations above is as follows.
From our study of $BCI(1)$ and from bounds already obtained (although for a different
model)~\cite{DGKRTZ10}, we already know that the asymptotic behaviour of the number of
lambda-terms widely differs from that of the number of trees: the signifiant increase in the number of lambda-terms of given size when compared
to Motzkin trees, i.e. the trees forming the underlying structure of lambda-terms, comes from the large numbers of ways of binding leaves to unary nodes; indeed we are dealing here with directed acyclic graphs.
Hence the r\^ole of the term $G_p^2$, which corresponds to the ``purely binary tree-like''
structure, is asymptotically negligible when compared to that of the differential term which
captures the binding of leaves.
\erem

\brem
The exact differential equation for $G_p(z)$ is \eqref{bci} whereas the arguments in
Remark~\ref{rem:3} show that
we may work with the linearized\footnote{This is not a linearization in a strict sense; we
did not replace the quadratic term by a linear one, but only omitted it.} equation 
\begin{equation} \label{lin_eq}
L_p(z) = C_{p-1} z^{2p} + \Delta_p L_p(z).
\end{equation} 
The linearized equation has a combinatorial interpretation as well; indeed, it counts the number of
structures $\cal S$ defined as follows: 
The smallest possible structures of $\cal S$ are precisely the smallest $BCI(p)$-terms, i.e., a
unary root followed by a binary tree with $2p-1$ nodes (and pointing to all leaves of this binary
tree). All terms in $\mathcal S$ have a unary node as
their root. To contruct larger terms, we add a new root and expand the subterm below using the
same edge hitting and expansion process as for $BCI(p)$-terms. Thus these terms may have binary
nodes, but never as root. 
\erem 

\begin{lemma}
\label{lemma:bounds}
For $p \geq 1$, the sequence $(\phi_n)_{n\ge 1}$ satisfies 
\[
2 \phi_1 \phi_{n-1} \leq
\sum_{l=1}^{n-1} \phi_{l} \phi_{n-1-l}  \leq
2 \phi_1 \phi_{n-1} + (n-3) \phi_{2} \phi_{n-2}.
\]
\end{lemma}
\begin{proof}
The lower bound is obvious: we just keep the first and the last term.
Set $q(n)=\phi_{n+1}/\phi_n$. To prove the upper bound, note that $(\phi_n)_{n\ge 1}$ is monotonically increasing 
and that for any $1\le i\le \lceil (n-3)/2\rceil$ we have 
\[
\phi_{2+i}\phi_{n-2-i}=\phi_{2} \phi_{n-2}\frac{q(2)q(3)\cdots q(1+i)}{q(n-2)q(n-3)\cdots q(n-1-i)}\ge \phi_{2} \phi_{n-2}.
\qedhere
\]
\end{proof}

Next we turn to the linearized equation \eqref{lin_eq}. 

\bth\label{lin_sol}
Set $\ell_{p,n}=[z^n]L_p(z)$ where $L_p$ is given by \eqref{lin_eq}. Then, for fixed $p$ and $n\to
\infty$, 
$$
\ell_{p,n}\sim B_p\beta_p^{n-1} n^{\gamma_p} (n-1)!^p
$$
where 
\begin{align} 
B_p&=C_{p-1}\prod_{j=1}^p \frac1{\Gamma\(1+\frac{2(p-k)-1}{2p+1}\)} \label{B_p} \\
&= C_{p-1} \exp\(-\frac{2p+1}2\int_1^2 \log(\Gamma(x))\,dx\)\(1+O\(\frac1p\)\), \quad\mbox{ as }
p\to\infty, \label{EML} \\
& \approx C_{p-1} (1.0844375142\dots...)^{(2p+1)/2}\(1+O\(\frac1p\)\) \nonumber
\end{align} 
and 
\begin{equation} \label{betagamma}
\beta_p=\frac{(4p+2)^p}{p!},\qquad  \gamma_p=\frac{p(p-2)}{2p+1}.
\end{equation} 
\eth

\bpf
Equation~\eqref{lin_eq} implies $\ell_{p,2p}=C_{p-1}$ and $\ell_{p,n}=Q_p(n-1)\ell_{p,n-2p-1}$ for
$n>2p$. Thus 
\begin{align} 
\ell_{p,(2p+1)n-1}&=C_{p-1} \prod_{j=1}^{n-1} Q_p(j) \nonumber \\ 
&= C_{p-1}\(\frac{(4p+2)^p}{p!}\)^{n-1} \prod_{k=1}^p
\frac{\Gamma\(n+\frac{2(p-k)-1}{2p+1}\)}{\Gamma\(1+\frac{2(p-k)-1}{2p+1}\)} \nonumber 
\\
&=C_{p-1}
\beta_p^{n-1} (n-1)!^p \prod_{j=1}^{n-1} \prod_{k=1}^p\(1+\frac {2(p-k)-1}{2p+1} \cdot\frac1j\)
\label{Q_p_prod}
\end{align} 
Finally, note that, as $n\to\infty$,  
$$
C_{p-1}\prod_{j=1}^{n-1}\prod_{k=1}^p\(1+\frac {2(p-k)-1}{2p+1} \cdot\frac1j\)\sim B_p n^\gamma_p.
$$
which completes the proof. The asymptotic form \eqref{EML} can be obtained by Euler-McLaurin's
formula. 
\epf

\begin{theorem}\label{bci_asympt}
For $p\geq 2$, the number of $BCI(p)$-terms of size $(2p+1)n-1$ is asymptotically 
\[
A_p \, \beta_p^{n-1} n^{\gamma_p} (n-1)!^p
\]
where $\beta_p$ and $\gamma_p$ are as in \eqref{betagamma} and $A_p=a_p B_p$ with $B_p$ as in
\eqref{B_p} and $a_p=1+O(1/(pe^p))$, as $p\to\infty$.
\end{theorem}

\brem
The first few values of the constants $a_p$ and $A_p$ appear in Table~\ref{tab}.
\erem

\begin{table}
\begin{center}
\begin{tabular}{|l|l|l|}
p & $a_p$ & $A_p$ \\ \hline 
2 & 1.048668\dots & 0.981017\dots \\
3 & 1.0046726194\dots & 2.19232485\dots \\
4 & 1.0006911656\dots & 6.17349476\dots \\
5 & 1.0001221936\dots & 19.2515312\dots \\
\end{tabular}
\end{center}
\caption{\small The first few values of $a_p$.} \label{tab}
\end{table}

\begin{remark}
Applying Stirling's formula we get the alternative form 
$$
\bar A_p \bar \beta_p^{n-1} n^{\bar \gamma_p} n^{np} 
$$
where 
$$
\bar \beta_p=\frac{\beta_p}{e^p},\qquad \bar\gamma_p=\frac{-5p}{4p+2}
$$
and $\bar A_p=(2\pi/e^2)^{p/2}A_p$. 
\end{remark}

\begin{proof}
From the recurrence relation for $\phi_n$, Eq.~\eqref{eq:recurrence-phi}, we have 
\begin{align*}
\phi_n &=
\phi_{n-1} Q_p(n-1) + \sum_{l=1}^{n-1} \phi_l \phi_{n-1-l}
\\ &=
\phi_{n-1} \, ( Q_p(n-1) + \Gamma_{n-1} ),
\end{align*}
with $\Gamma_{n-1} = \sum_{1\leq l \leq n-1} \phi_l \phi_{n-1-l} / \phi_{n-1}$ and $Q_p(n)$ defined
in \eqref{eq:defQ(n)}. 
Thus 
\[
\phi_n =
\phi_1 \prod_{j=1}^{n-1} (Q_p(j) + \Gamma_j) =
K_p(n) \phi_1 \prod_{j=1}^{n-1} Q_p(j) 
\]
where $K_p(n)= \prod_{j=1}^{n-1}\left( 1+\frac{\Gamma_j}{Q_p(j)} \right)$. For $p\ge 2$ we have
$Q(n)=\Omega(n^p)$ and 
furthermore Lemma~\ref{lemma:bounds} gives $\Gamma_{n-1} = 2 \phi_1 + O(1/n^{p-1})=2 C_{p-1} +
O(1/n^{p-1})$. Hence 
the sequence $(K_p(n))_{n\ge 1}$ is convergent and we get 
\[
\phi_n = a_p C_{p-1}\(\prod_{j=1}^{n-1} Q_p(j)\) \(1+O\(\frac1n\)\)
\]
where $a_p=C_{p-1}\cdot \lim_{n\to\infty} K_p(n)$.
The product $C_{p-1}\prod_{j=1}^{n-1} Q_p(j)$ is already evaluated in \eqref{Q_p_prod}, yielding
the asymptotic behaviour of the solution of the linearized equation given in
Theorem~\ref{lin_sol}.


The difference between the linearization and the $\phi_n$ is hidden in the constant $a_p$. Thus we
are left with the determination of $a_p$. We will confine ourselves with an asymptotic evaluation
for $p\to\infty$. 

First note that Lemma~\ref{Q_p} immediately implies the inequality 
\begin{equation} \label{Q_inequ}
Q_p(n)\ge \frac{2^p(2p+1)^p}{p!}n^p. 
\end{equation} 
Now observe that $\Gamma_1=0$ and that
by Lemma~\ref{lemma:bounds} we have $\Gamma_j\le 2\phi_1+(j-2)\phi_2\phi_{j-1}/\phi_j$. The
quotient in the last term was already estimated in the proof of
Lemma~\ref{lemme:maj-coefficients} by $\phi_{j-1}/\phi_j\le 1/Q_p(j)$. 
Using this estimate as well as the inequality \eqref{Q_inequ} we obtain (for $j>1$) 
$$
\Gamma_j\le 2\phi_1 + j\frac{\phi_2 p!}{2^p(2p+1)^p j^p} = 2C_{p-1} + j\frac{\phi_2 p!}{2^p(2p+1)^p
j^p}.
$$
Hence we get 
\begin{align} 
a_p&= \prod_{j\ge 2}\( 1+\frac{\Gamma_j}{Q_p(j)} \) \nonumber \\
&\le  \prod_{j\ge 2}\( 1+\frac{C_{p-1} p!}{2^p(2p+1)^p j^p} + \frac{\phi_2
p!}{2^{2p}(2p+1)^{2p} j^{2p-1}}\) \nonumber \\
&\le \prod_{j\ge 2}\( 1+\frac{C_{p-1} p!}{2^p(2p+1)^p j^p}\) \prod_{j\ge 2}\( 1+\frac{\phi_2
p!}{2^{2p}(2p+1)^{2p} j^{2p-1}}\).  \label{prods}
\end{align} 
The two products above turn out to be of the form $\prod_j \(1+\frac{\eps_p}{j^p}\)$ with $\eps
\to 0$ as $p\to\infty$. Thus we can easily estimate them by 
\begin{align*} 
\log\prod_j \(1+\frac{\eps_p}{j^p}\)
&=\sum_j \sum_{k\ge 1} \frac{(-1)^{k-1}}k \frac{\eps_p^k}{j^{pk}} 
=\sum_k \frac{(-1)^{k-1}}k \eps_p^k \zeta(pk). 
\end{align*} 
Since $\zeta(x)=1+O(2^{-x})$ as $x\to\infty$ we obtain $\prod_j \(1+\frac{\eps_p}{j^p}\)=1
+O(\eps_p)$. Now turning to \eqref{prods} we have 
$$
\frac{C_{p-1} p!}{2^p(2p+1)^p j^p} \sim \frac1{pe^p\sqrt{2e}} \qquad\mbox{ and }\qquad \frac{\phi_2
p!}{2^{2p}(2p+1)^{2p} j^{2p-1}} =o (e^{-2p}) 
$$
where we used $\phi_2=\phi_1 Q_p(1)=C_{p-1} 4^p\binom{2p-1}p$ for the second estimate. 
This implies $a_p = 1+O(1/(pe^p))$ which completes the proof.
\epf

\section{Closed lambda-terms}\label{clolamter}

So far, we are unable to determine the asymptotic behaviour of $\lambda_n$. We will derive upper
and lower estimates and a recurrence relation which allows an efficient computation of $\lambda_n$.

\subsection{Estimates for $\lambda_n$} 

The number of $BCI(p)$-terms is certainly a lower bound, but using rather crude and elementary
estimates a better bound can be obtained.

\bth
The number $\lambda_n$ of closed lambda-terms of size $n$ satisfies for every $\eps>0$ and for
sufficiently large $n$ the inequalities 
$$
c_1\(\frac{4n}{e\log n}\)^{n/2} \frac{\sqrt{\log n }}n\le \lambda_n \le c_2
\(\frac{9(1+\eps) n}{e\log n}\)^{n/2}\frac{(\log n)^{n/(2\log n)}}{n^{3/2}}
$$
where $c_1,c_2$ are some positive constants.
\eth

\bpf
We determine the lower bound by counting particular lambda-terms of size $n$. 
Take a binary tree with $n_f$
leaves and attach to its root a string of $n_u$ unary nodes. Then connect the leaves to the unary
nodes by pointers. Each such object is a closed lambda-term and there are $C_{n_f} n_u^{n_f}$ such
terms. Note that $n_u=n+1-2n_f$. Hence we obtain 
$$
\lambda_n\ge \sum_{n_u=1}^{n-1} C_{n_f} n_u^{(n+1-n_u)/2} \ge C_{\tilde n_f} n_{\tilde
u}^{(n+1-\tilde n_u)/2}
$$
where $\tilde n_u$ and $\tilde n_f$ are those values of $n_u$ and $n_f$, respectively, where
$n_u^{n_f}$ attains its maximum. The maximum is attained at
$\tilde n_u=n/W(en)$ where $W(n)$ is Lambert's $W$-function defined implicitly by
$W(n)e^{W(n)}=n$. It is easy to show that 
$$
W(en)= \log n -\log\log n +1 +O\(\frac{\log\log n}{\log n}\).
$$
This implies 
\begin{equation} \label{n_u}
\frac{n}{\log n} \le \tilde n_u \le \frac{n}{\log n-\log\log n}. 
\end{equation} 
Hence we obtain 
\begin{align*} 
\tilde n_u^{\tilde n_f}
&\ge \(\frac{n}{\log n}\)^{(n/2)\cdot (1-1/(\log n-\log\log n))+1/2} \\
&= \(\frac{n}{\log n}\)^{n/2} \sqrt{\frac{n}{\log n}} 
\exp\(-\frac n{2(\log n-\log\log n)} (\log
n-\log\log n)\). 
\end{align*} 
The lower estimate now follows from $C_r\sim k_1 4^r/r^{3/2}$ ($r\to\infty$) where $k_1$ is some
positive constant.

For the upper estimate we construct a set of objects such that a proper subset corresponds to the
set of all lambda-terms of size $n$. Take a Motzkin tree and add pointers such that each leaf is
connected to an arbitrary unary node. Clearly, each lambda-term is generated in that way. But
since leaf $x$ might be bound to a unary node which is not on the path from $x$ to the root, we
generate also enriched trees which do not represent a lambda-term. Therefore we get the upper
bound $\lambda_n\le M_n \max n_u^{n_f}$ where $M_n$ is the number of Motzkin trees with $n$
vertices. As above we have $n_u=n/W(en)$. Now \eqref{n_u} implies that for sufficiently
large $n$ we have 
\begin{align*} 
n_u^{n_f}&\le \(\frac{n}{\log n-\log\log n}\)^{\frac n2\(1-\frac1{\log n}\)} \\
&\le \(\frac{(1+\eps)n}{\log n}\)^{\frac n2} \(\frac{n}{\log n}\)^{-\frac{n}{2\log n}} \\
&=\(\frac{(1+\eps)n}{e\log n}\)^{\frac n2}\exp\(\frac {n\log\log n}{2\log n}\)
\end{align*} 
where we used $\log n/(1+\eps)\le \log n -\log\log n$ for sufficiently large $n$. Finally, the
well known fact $M_r\sim k_2 3^r /r^{3/2}$ (as $r\to\infty$ and with some constant $k_2>0$)
completes the proof.
\epf

\brem
If $\bar\lambda_n$ is the number of closed lambda-terms where the sum of the number of unary nodes
and the number of binary nodes equals $n$ (so leaves do not contribute to the size), then David et
al. \cite{DGKRTZ10} showed the following result for the growth rate of the counting sequence: 
\[
\left(  \frac{(4-\epsilon) n}{\log n}   \right)^{n - n/\log n}
\leq \bar\lambda_n \leq
\left(  \frac{(12+\epsilon) n}{\log n}   \right)^{n - n/3\log n}.
\]
Thus the exponential growth is similar to that of $\lambda_n$, although the underlying model is
rather different. 
\erem

\subsection{A recurrence relation}

Eq. \eqref{direct} immediately implies that $\lambda_n$ satisfies the recurrence relation   
\begin{equation} \label{reclambda}
\lambda_n = M_{n-1} + \sum_{\ell+q=n-1} \lambda_\ell \lambda_q + \sum_{1 \leq \ell \leq n-1} \delta_{n,\ell}
\lambda_\ell
\end{equation} 
where $M_n=[z^n]M(z,1)$ is the number of Motzkin trees of size $n$ and 
\[
\delta_{n,\ell}=[ z^{n-1-\ell} ] \frac{1}{(1-2zM(z))^\ell} = \sum_{r \geq 0} \binom{\ell-1+r}{\ell-1}
\zeta_{n-\ell-1,r}
\]
with $\zeta_{s,r} := [ z^s ] ( 2z M(z))^r$. Note that $\zeta_{s,r}=0$ unless $s\ge 2r$ and thus
\[
\delta_{n,\ell}=\sum_{r=0}^{\lfloor (n-\ell-1)/2 \rfloor } \binom{\ell-1+r}{\ell-1} \zeta_{n-\ell-1,r}.
\]
By Lagrange inversion we obtain 
\[
\zeta_{s,r}=2^r [ z^{s-r} ] M(z)^r 
=2^r \frac{r}{s-r} \sum_{a,b,c: \; b+2c = s-2r} \binom{s-r}{a,b,c}
\]
which gives after a few computations 
\begin{equation} \label{deltalambda}
\delta_{n,\ell}=\sum _{t=0}^{ \left\lfloor \frac{n-\ell-1}{2} \right\rfloor} \sum_{r=0}^t {\frac
{r{2}^{r}\binom{\ell-1+r}{r} \left( n-\ell-2-r \right) !}{t!\,
 \left( t-r \right) !\, \left( n-\ell-1-2\,t \right) !}}
\end{equation} 
Now, set $b_{n,\ell,t}:=\sum_{r=0}^t {\frac {r{2}^{r}\binom{\ell-1+r}{r} \left( n-\ell-2-r \right) !}{t!\,
 \left( t-r \right) !\, \left( n-\ell-1-2\,t \right) !}}$. This sum is amenable to creative
telescoping (see \cite{Ze91}) which yields a system of two recurrences of order one for the
multi-index sequence $(b_{n,\ell,t})_{n,\ell,t\ge 0}$: 
\begin{align*} 
&\left( -{\ell}^{2}-2nt-2\ell t-\ell-n+{n}^{2} \right) b_{n,\ell,
t} + \left( 2\ell t-2\,n\ell+2{\ell}^{2}+4\ell \right) b_{n,\ell+1,t} \\
&\qquad +\left( -4{t}^{2}-2t+4nt-4\ell t-{n}^{2
}+2n\ell-\ell+n-{\ell}^{2} \right) b_{n+1,\ell,t}=0
\end{align*} 
and
\begin{align*} 
 &\left( 2n-t-2 \right)  \left( n-\ell-2t-2 \right)  \left( n-\ell-2t-1
 \right) b_{{\ell,n,t}} 
-t \left( t+1 \right) 
 \left( n-\ell-t-2 \right) b_{{\ell,n,t+1}}
\\
&\qquad 
-( n-\ell-2t-2 )  ( n-\ell-2t-1) \left( n-\ell-2t \right) b_{{\ell,n+1,t}} =0.
\end{align*} 
with the initial conditions given by the sum representation of $b_{n,\ell,t}$. This system can be
solved explicitly and we get
$$
b_{n,\ell,t}=\dfrac{2\ell}{t}\cdot
\dfrac{\Gamma\left(n-\ell-2\right){\;_2F_1(-t+1,\ell+1;-n+\ell+3;2)}}{\Gamma\left(t\right)^{2} 
\Gamma\left(n-\ell-2t\right)}
$$
Now, using the same techniques, we can also obtain a system of two D-finite recurrences for 
$\delta_{n,\ell}$:
\begin{align*} 
&  \left( n-\ell \right)  \left( n+1-\ell \right)  \left( 
n-2\ell-2 \right) \delta_{{n+2,\ell}} 
-\left( n-\ell \right)  \left(2{n}^{2}-6n
\ell-5n+ 2{\ell}^{2}+3\ell+1 \right) \delta_{{n+1,\ell}}
\\
&\qquad
-\left( n-1 \right)  \left( 3{n}^{2}-2n\ell+n-{\ell}^{2}-9\ell-8
 \right) \delta_{{n,\ell}}
+20 \left( n-1 \right) \ell \left( \ell+1
 \right) \delta_{{n,\ell+2}}
\\
&\qquad
+2 \left( n-1 \right)  \left( 5n-9\ell-12
 \right) \ell \delta_{{n,\ell+1}}=0
\end{align*} 
and
\begin{align*} 
& \left( n-\ell \right)  \left( \ell-n-1 \right) \delta_{{
n+2,\ell}}+ \left( n-\ell \right)  \left( 2n-\ell \right) \delta_{{n+1,\ell}}
-\ell \left( n-1 \right) \delta_{{n+1,\ell+1}}
\\
&\qquad
-4\ell \left( n-1 \right) \delta_{{n,\ell+1}}
 \left( n-1 \right)  \left( 3n-2\ell+1 \right) \delta_{{n,\ell}}=0.
\end{align*} 
Unfortunately, this system does not admit an explicit solution in terms of classical special
functions. Nevertheless, it allows us to calculate efficiently the values $\delta_{n,\ell}$.

\section{Conclusion and outlook}\label{concl}


The motivation for our analysis was the enumeration of closed lambda-terms. Since the problem
seems hard, we treated the subclass of $BCI(p)$-terms which imposes quite a restriction on the
degrees of freedom in binding variables by quantifiers. Thus we expected the set of $BCI(p)$-terms
to be small in comparison to the set of closed lambda-terms. Our results verify and quantify this. 
Moreover, they show that the restriction is weaker than bounding the unary height, i.e., the
maximal number of unary nodes on a root-to-leaf path. Indeed, if the unary height is bounded by
$L$, then the asymptotic number of terms of size $n$ is in general $C n^{-3/2} \rho^n$;
i.e. the asymptotic behaviour is like that of the number of Motzkin trees. Only for particular
values of~$L$, the asymptotic behaviour becomes $C n^{-5/4} \rho^n$ (see \cite{BGG11}). This
behaviour changes if the condition on the unary height is replaced by a condition on the number of
pointers per unary node (as in the $BCI(p)$ case) or dropped completely (closed lambda-terms). So
these structures are indeed different from tree-like structures; their counting sequences grow
much faster than that of Motzkin trees. So we can conclude that the enumeration of $BCI(p)$-terms
is not only of interest in its own right, but also more closely related to the original counting
problem than to tree enumeration.  

Since the union of all the sets of $BCK(p)$-terms, $p=1,2,\dots,$ is precisely the set of
closed lambda-terms, one might be tempted to approach the problem of determining the
asymptotic behaviour of $\lambda_n$ via the number of $BCK(p)$-terms and letting $p\to \infty$. 
For performing such a limit we needed precise and uniform asymptotics for the number of
$BCK(p)$-terms. Unfortunately, the asymptotic computation of the number of $BCK(p)$-terms turns out to be 
much more involved than that of the number of $BCI(p)$-terms. A precise analysis of the $BCK$ case
is beyond the scope of this paper and will be the topic of a forthcoming paper. Here we discuss
only briefly how to attack this problem. 

The differential equation \eqref{bck} implies a recurrence relation for the coefficients of
$F_p(z)$. This can be linearized in a similar fashion as we did in the $BCI$ case (essentially
Lemmas~\ref{lemma:extremal-terms}-\ref{lemma:bounds}). The next step will be showing upper and
lower estimates for $f_n:=[z^n]F_p(z)$. This enables us to identify the asymptotically dominant
term in the recursion which yields a rough information on the growth of $f_n$. 

The task is now to find the asymptotic behaviour of the correct solution. 
The growth rate of the coefficients tells us that the Borel transform $\hat F_p(z)$ of the
generating function $F_p(z)$ must grow exponentially in $z$. This
indicates that $\hat F_p(z)$ is Hayman-admissible (\emph{cf.} \cite{Hayman}) and therefore a saddle point analysis applies and eventually yields the asymptotic number of $BCK(p)$-terms. 

When studying not only the size but further properties of $BCK(p)$-terms by means of multivariate
generating functions, the above remarks suggest that these functions will be (multivariate) 
Hayman-admissible such that a multivariate saddle point method applies (\emph{cf.} \cite{GM06}).

As in the case of closed lambda-terms, the functional equation \eqref{bck} corresponds to a
recurrence relation of the form \eqref{reclambda}. The only difference is that $\delta_{n,l}$ in
\eqref{deltalambda} has to be replaced by 
$$
\delta_{n,l}=\sum _{t=0}^{ \min\(p,\left\lfloor \frac{n-l-1}2\right\rfloor\)} \sum_{r=0}^t {\frac
{r{2}^{r}\binom{l-1+r}{r} \left( n-l-2-r \right) !}{t!\,
 \left( t-r \right) !\, \left( n-l-1-2\,t \right) !}}.
$$
Similarly as before, this gives rise to a system of D-finite recursions.

\begin{ack}
The authors thank Marek Zaionc for triggering the authors' interest in the subject and for numerous
fruitful discussions about it.
\end{ack}

\bibliography{bci}

\end{document}